\documentclass[subeqn, letterpaper, reqno, 11pt]{amsart}

\usepackage{latexsym}
\usepackage{amssymb}
\usepackage{amsmath}
\usepackage{amsfonts}
\usepackage[mathscr]{eucal}
\usepackage{times}
\usepackage{enumerate}

\usepackage{amscd}
\usepackage{tkz-euclide}

\usepackage[numbers, sort&compress]{natbib}

\usepackage{psfrag}

\usepackage{amsthm}
\usepackage[all,cmtip]{xy}
\usepackage{hyperref}
\usepackage{comment}

\input{xypic}

\usepackage{tikz}
\usetikzlibrary{backgrounds,calc}
\usetikzlibrary{calc,arrows}

\newtheorem{theorem}{Theorem}[section]
\newtheorem{thm}{Theorem}[section]
\newtheorem{conj}{Conjecture}[section]

\newtheorem{lemma}[thm]{Lemma}

\newtheorem{prop}[thm]{Proposition}

\theoremstyle{definition}
\newtheorem{definition}[thm]{Definition}
\newtheorem{defn}[thm]{Definition}

\newtheorem{remark}[thm]{Remark}

\numberwithin{equation}{section}

\newcommand{\R}{\mathbb{R}}

\DeclareMathOperator{\dist}{dist}

\newcommand{\hind}{\mathrm{ind}}
\newcommand{\dbl}{\mathrm{dbl}}

\begin{document}



\title{Bounded distance geodesic foliations in Riemannian planes}


\author[J.~Ge]{Jian Ge*}
\address[Ge]{Beijing International Center for Mathematical Research, Peking University, Beijing 100871, P. R. China.}
\email{jge@math.pku.edu.cn}
\thanks{*Partially supported by NSFC}


\author[L.~Guijarro]{Luis Guijarro**}
\address[Guijarro]{ Department of Mathematics, Universidad Aut\'onoma de Madrid, and ICMAT CSIC-UAM-UCM-UC3M, Spain.}
\curraddr{}
\email{luis.guijarro@uam.es}
\thanks{**Supported by research grants  
MTM2017-85934-C3-2-P from the MINECO, by ICMAT Severo Ochoa Project SEV-2015-0554 (MINECO) and by the Severo Ochoa Programme for Centers of Excellence in R\&D, CEX2019-000904-S (MICINN)}

\date{\today}




\begin{abstract}
A conjecture of  Burns-Knieper \cite{BK1991} asks whether a 2-plane with a metric without conjugate points, and with a geodesic foliation whose lines are at bounded Hausdorff distance, is necessarily flat. We prove this conjecture in two cases: under the hypothesis that the plane admits total curvature,  and under the hypothesis of visibility at some point. Along the way, we show that  all geodesic line foliations on a Riemannian 2-plane must be homeomorphic to the standard one.

\end{abstract}
\maketitle




\section{Introduction}
A complete Riemannian manifold $(M, g)$ is called without conjugate points if no geodesic on $M$ contains a pair of mutually conjugate points. The class of Riemannian manifolds without conjugate points contains the set of Riemannian manifolds with nonpositive curvature as a proper subset. E. Hopf \cite{Hop1948} proved that any metric without conjugate points on the 2-dimensional torus $T^{2}$ is necessarily flat. In fact, E. Hopf proved that the total curvature of such manifold must be non-positive, hence by the Gauss Bonnet Theorem, the metric must be flat. In \cite{Gre1958}, Green generalized Hopf's idea to higher dimensions  by showing that in any \emph{closed} Riemannian manifold without conjugate points, the integral of its scalar curvature is  nonpositive .

Although there are examples of manifolds without conjugate points having positive sectional curvatures at some point, there are many properties holding for non-positively curved metrics that also hold for metrics without conjugate points. However this is not always the case. Recall that a smooth curve $\gamma: (-\infty, \infty)\to M$ in a Riemannian manifold $M$ is called a line if $\dist(\gamma(t), \gamma(s))=|t-s|$ for all $t, s\in \mathbb R$. The Flat Strip Theorem for a simply connected surface $M$ of nonpositive curvature states that any two lines $\sigma$ and $\gamma$ in $M$ with bounded Hausdorff distance must bound a flat strip, i.e., there exists an isometric embedding from $I\times \mathbb R$ to $M$ for some closed interval $I\subset \mathbb R$ such that $\sigma$ and $\gamma$ are the boundaries of the image. However in \cite{Bur1992}, Burns constructed an example showing that the flat strip theorem does not hold for metric without conjugate points. Nonetheless, his construction does not violate the following conjecture proposed in \cite{BK1991}:
\begin{conj}\label{conj:main}
Let $M$ be a simply connected surface with a complete Riemannian metric with no conjugate points. Suppose that $\mathcal F$ is a foliation on $M$ whose leaves are all geodesic lines, and any two of which are at finite Hausdorff distance. Then $M$ is flat.
\end{conj}
In \cite{BK1991}, the authors proved the flatness of $M$ under the extra assumption that for any given line $\gamma$ there exists a bounded Hausdorff distance line foliation $\mathcal F$ containing $\gamma$.

In this paper, we confirm the above conjecture 
in two different cases. 
 
\begin{theorem}\label{thm:flat}
If $M$ admits  total curvature, then Conjecture \ref{conj:main} holds.
\end{theorem}

\begin{theorem}\label{thm:visib_foliation}
If $M$ satisfies the visibility axiom from a given point $p$, then $M$ does not admits a line foliation with bounded Hausdorff distance. In particular if $M$ is the universal covering of a closed surface, then Conjecture \ref{conj:main}  holds.
\end{theorem}

Our approach to the conjecture can be divided into two parts. In the first part, we study the topology of the foliation $\mathcal F$. General foliations on Riemannian surfaces by curves could be very complicated, cf \cite{Kap1940, Kap1941, HR1957},  even after assuming that the leaves are geodesics. However, if we assume the leaves are lines, we have
\begin{theorem}\label{thm:foliation}
Let $(M^{2}, g)$ be a complete Riemannian surface homeomorphic to $\mathbb R^{2}$, and let $\mathcal F$ be a geodesic line foliation on $M$. Then $\mathcal F$ is homeomorphic to the standard straight line foliation on the Euclidean plane $\mathbb R^{2}$.
\end{theorem}
Note that \autoref{thm:foliation} does not require a metric without conjugate points on $M$, so it might have an independent interest and could  be applied to other situations.

The second part of our approach consists of a close study of compactifications of $M$ by the introduction of equivalence relations on the set of rays, which is an effective tool to control the global geometry of $M$. Such a tool has already proved to be useful in \cite{GGS2019}. In this paper, we pursue this direction further. We need the following pair of definitions.

\begin{definition}
Let $M$ be a complete Riemannian surface homeomorphic to $\mathbb R^{2}$. Let $\gamma_{i}:\mathbb R\to M$, $i=1, 2$ be two non-intersecting lines in $M$.  Then $\gamma_{1}$ and $\gamma_{2}$ bounds a region homeomorphic to the band $\mathbb R\times [0, 1]$, which has two ends. By changing the orientation of $\gamma_{2}$ if necessary, we assume $\gamma_{1}(t)$ and $\gamma_{2}(t)$ belong to the same end of the band as $t\to +\infty$. 

\begin{enumerate}
\item
We call $(\gamma_{1}, \gamma_{2})$ a \emph{ hyperbolic pair } if 
$$\lim_{t\to+ \infty} d(\gamma_{1}(t), \gamma_{2}(t))=\infty, \qquad
\lim_{t\to -\infty} d(\gamma_{1}(t), \gamma_{2}(t))=\infty,
$$

\item
We call $(\gamma_1,\gamma_2)$ a \emph{weak hyperbolic pair} if the above inequalities hold for some sequence $t_k\to\infty$. 
\end{enumerate}
\end{definition}

It is clear that a hyperbolic pair is a weak hyperbolic pair, but the reciprocal does not need to hold.
The importance of this definition lies in the following  Propositions.

\begin{prop}\label{prop:total_hyp}
Let $M$ be a complete Riemannian surface homeomorphic to $\mathbb R^{2}$ without conjugate points. Suppose the total curvature of $M$ exists. If $c(M)<0$, then for any line $\gamma$ in $M$,  there exists a line $\sigma$ such that $(\gamma ,\sigma)$ is a hyperbolic pair.
\end{prop}

\begin{prop}\label{prop:visi_hyp}
Let $M$ be a complete Riemannian surface homeomorphic to $\mathbb R^{2}$ without conjugate points satisfying the visibility axiom from some point $p$. Then for any line $\gamma$ in $M$,  there exists a line $\sigma$ such that $(\gamma ,\sigma)$ is a weak hyperbolic pair.
\end{prop}

Once we have the existence of (weak) hyperbolic pairs, its combination with Theorem \ref{thm:foliation} will give relatively simple proofs of the conjecture.

\subsection{Some properties of planes without conjugate points}
\label{subsec:ncp_properties}
We will collect here a few properties of Riemannian planes without conjugate points that, although elementary, will be used repeatedly in the paper. Recall that a line is a geodesic $\gamma:\R\to M$ such that it minimizes the distance between any two of its points. 

\begin{lemma}
\label{lem:line properties}
Let $M$ be a complete Riemannian plane without conjugate points. Then
\begin{enumerate}
\item Any geodesic $\gamma:\R\to M$ is a line;
\item any pair of lines $\gamma_1$, $\gamma_2:\R\to M$ are disjoint or intersect at most in a single point;
\item each line $\gamma:\R\to M$ divides $M$ into two closed geodesically convex sets. 
\end{enumerate}
\end{lemma}
\begin{proof}
Using standard arguments, one shows that the lack of conjugate points implies that the exponential map is a diffeomorphism at every point. This shows the first and second points. The third follows immediately from the second.  
\end{proof}

\section{Line foliations on a Riemannian plane}

In this section we assume that $M$ is a Riemannian plane, and  with a given foliation $\mathcal F$ by geodesic lines.
For any leaf $\sigma\in \mathcal F$, $\sigma$ separates $M$ into two closed half spaces, that we denote as $H_{\sigma}^{\pm}$. For any two distinct leaves $\sigma, \gamma\in \mathcal F$, there is a half space defined by $\sigma$ containing $\gamma$, say $H_{\sigma}^{-}$. On the other hand there is a half space defined by $\gamma$ containing $\sigma$, say $H_{\gamma}^{+}$. 
\begin{defn}
The region bounded by $\sigma$ and $\gamma$ is called a \emph{strip}, denoted by
$$
S_{\gamma}^{\sigma}=H_{\sigma}^{-}\cap H_{\gamma}^{+}.
$$
\end{defn}

We need the following lemma to show all those strips between two leaves behave well.

\begin{lemma}\label{lem:betweenness}
Let $\sigma_{i}$, $i=1,3$ be two leaves in $\mathcal F$, $p$ be an interior point of the strip $S_{\sigma_{1}}^{\sigma_{3}}$. Let $\sigma_{2}$ be the leaf passing through $p$. Then $\sigma_{2}$ separates $S_{\sigma_{1}}^{\sigma_{3}}$, i.e.
$$S_{\sigma_{1}}^{\sigma_{2}}\cap S_{\sigma_{2}}^{\sigma_{3}}=\sigma_{2},$$ and
$$S_{\sigma_{1}}^{\sigma_{3}}=S_{\sigma_{1}}^{\sigma_{2}}\cup S_{\sigma_{2}}^{\sigma_{3}}.$$
\end{lemma}
\begin{proof}
Suppose the contrary holds, i.e. $\sigma_{2}$ does not separate $S_{\sigma_{1}}^{\sigma_{3}}$. Schematically, it looks like \autoref{pic:tripod}. 

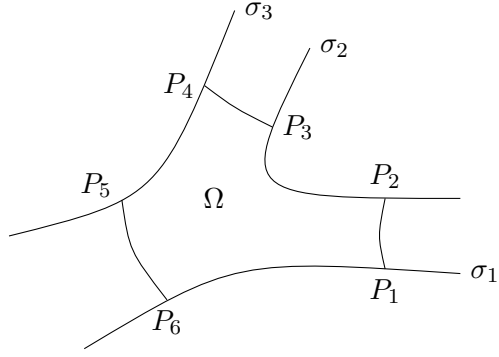
\begin{figure}
\begin{tikzpicture}
\draw(0,1.5)..controls(2,2)..(3,4.5)node[right]{$\sigma_{3}$};
\draw(1,0)..controls(3,1.2)..(6,1)node[right]{$\sigma_{1}$};
\draw(6,2)..controls(3,2)..(4,4)node[right]{$\sigma_{2}$};
\draw (1.5,1.98)node[left, yshift=0.2cm]{$P_{5}$}..controls(1.6,1.3)..(2.1,0.65)node[below]{$P_{6}$};
\draw (5,1.06)node[below]{$P_{1}$}..controls(4.9,1.5)..(5,2)node[above]{$P_{2}$};
\draw (2.6,3.5)node[left]{$P_{4}$}..controls(3,3.2)..(3.5,2.95)node[right]{$P_{3}$};
\draw (3,2)node[left]{$\Omega$};
\end{tikzpicture}
\caption{Construction of the region $\Omega$}\label{pic:tripod}
\end{figure}

Then we can choose six points $P_{i}, i=1,\cdots, 6$ on the leaves as shown in the picture. Here is the detailed construction of choosing such points. For arbitrary chosen $P_{5}\in \sigma_{3}$ and $P_{6}\in \sigma_{1}$, we let $\overline{P_{5}P_{6}}$ be a shortest geodesic connecting them.
Clearly 
$$\overline{P_{5}P_{6}} \cap \sigma_{2}=\varnothing$$
because geodesic segments do not intersect within their interiors.

The $\overline{P_{5}P_{6}}$ separates $S_{\sigma_{1}}^{\sigma_{3}}$ into two parts:
\begin{equation}\label{eq:r1}
S_{\sigma_{1}}^{\sigma_{3}}=S_{\sigma_{1}}^{\sigma_{3}}(-)\cup S_{\sigma_{1}}^{\sigma_{3}}(+).
\end{equation}
Under the assumption that $\sigma_{2}$ does not separate $S_{\sigma_{1}}^{\sigma_{3}}$,  one of the regions in \eqref{eq:r1}, say $S_{\sigma_{1}}^{\sigma_{3}}(+)$, must contain $\sigma_{2}$. We choose any $P_{1}\in \sigma_{1}\cap S_{\sigma_{1}}^{\sigma_{3}}(+)$ which differs from $P_{6}$. We then choose an arbitrary $P_{2}\in \sigma_{2}$. By the same argument, the segment $\overline{P_{1}P_{2}}$ separates the strip $S_{\sigma_{1}}^{\sigma_{2}}$ into two parts, one of which contains $\sigma_{3}$, say $S_{\sigma_{1}}^{\sigma_{3}}(-)$. We then choose arbitrary 
$$P_{3}\in \sigma_{2}\cap S_{\sigma_{1}}^{\sigma_{3}}(-)\cap S_{\sigma_{1}}^{\sigma_{3}}(+),$$
$$P_{4}\in \sigma_{3}\cap S_{\sigma_{1}}^{\sigma_{3}}(-)\cap S_{\sigma_{1}}^{\sigma_{3}}(+),$$
which differs from $P_{2}$ and $P_{5}$.  Since geodesic segments do not intersect more than once in their interior, the geodesic segments  $\overline{P_{1}P_{2}}$, $\overline{P_{3}P_{4}}$ and $\overline{P_{5}P_{6}}$ do not intersect. Then we denote the compact region in $M$, bounded by the broken geodesic loop $\overline{P_{1}P_{2}P_{3}P_{4}P_{5}P_{6}P_{1}}$, by $\Omega$. Clearly $\Omega$ is homeomorphic to the 2-dimensional disk. 
 
By the definition of foliation and non-tangency of geodesic in Riemannian manifolds, through each point on $\overline{P_{1}P_{2}}$, $\overline{P_{3}P_{4}}$ and $\overline{P_{5}P_{6}}$, there is a unique leaf passing through it. Therefore, there is a well defined nowhere singular line field $\xi$ given by the foliation $\mathcal F$ on $\Omega$. Let $\dbl(\Omega)$ be the double of $\Omega$, i.e. the gluing of two copies of $\Omega$ along the common boundary. Clearly $\dbl(\Omega)$ is homeomorphic to $S^{2}$. We can glue the two foliations together along the boundary to get a nowhere singular line field on $\dbl(\Omega)\setminus\{P_{1},P_{2},P_{3},P_{4},P_{5},P_{6}\}$. Denote the line field on $\dbl(\Omega)$ by $\bar\xi$. The local picture of $\bar\xi$ looks like \autoref{pic:hopf} at the conner point $P_{i}$.

\begin{figure}
\begin{tikzpicture}
\draw (0,2.5)--(0,3);
\draw (-2,3)..controls(0,0)..(2,3);
\draw (-1.5,3)..controls(0,.5)..(1.5,3);
\draw (-1,3)..controls(0,1)..(1,3);
\draw (-0.5,3)..controls(0,1.5)..(0.5,3);
\end{tikzpicture}
\caption{Local picture of $\bar\xi$ near $P_{i}$}\label{pic:hopf}

\end{figure}
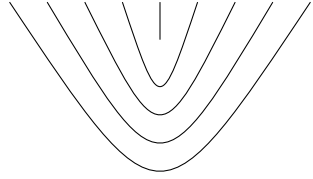

By the Hopf Index Theorem, we have
$$
\sum_{i=1}^{6} \hind_{\bar\xi}(P_{i})=\chi(\dbl(\Omega)),
$$
where $\hind_{\bar\xi}(P_{i})$ is the Hopf index of the line field $\bar\xi$ at the singularity $P_{i}$ and $\chi(\dbl(\Omega))$ is the Euler characteristic number of $\dbl(\Omega)$.  The Euler number is clearly equals to $2$ since the double is homeomorphic to the 2-dimensional sphere. On the other hand by \cite{hopf} page 109, the index of $\bar\xi$ at $P_{i}$ is $\frac12$. Therefore $\sum_{i=1}^{6} \hind_{\bar\xi}(P_{i})=3$, we get a contradiction.
\end{proof}

\begin{proof}[Proof of \autoref{thm:foliation}]
In \cite{HR1957}, it is seen that by passing to their quotient spaces, isomorphism classes of foliations in $\mathbb R^{2}$ with non-compact leaves are in one-to-one correspondence with the set of simply connected (not necessarily Hausdorff) one-dimensional manifolds without boundary with a countable basis. 
Thus, to prove our theorem, we only need to show that the leaf space of our line foliation $\mathcal{F}$ is homeomorphic to the real line $\mathbb R$. We will do this by showing that $M/\mathcal{F}$ is  Hausdorff. 

First, for any two leaves $\sigma_{-1}$ and $\sigma_{1}$, we can pick a geodesic segment (not necessarily unit speed) $\gamma:[-2,2]$ connecting them, i.e. $\gamma(i)\in \sigma_{i}$ for $i=-1, 1$. By Lemma \autoref{lem:betweenness}, for each $t\in (-2,2)$ there is a leaf $\sigma_t$ in $\mathcal{F}$ passing through $\gamma(t)$. Once again, Lemma \ref{lem:line properties} gives that    $\sigma_t$  and $\gamma$ will intersect in a single point.

Therefore the sets $\{\sigma_{t}\,:\, t\in (-2, 0)\}$ and $\{\sigma_{t}\,:\, t\in (0, 2)\}$ serve respectively as separating neighborhoods for the two leaves $\sigma_{-1}$ and $\sigma_{1}$. This proves that the leave space is Hausdorff; by \cite{HR1957}, the leave space is homeomorphic to $\mathbb R$ and our foliation is equivalent to the standard foliation of $\R^2$.
\end{proof}

\section{Total curvature and hyperbolic pairs}
In this section we review basic properties of the ideal boundary of a surface admitting total curvature, and prove  Proposition \ref{prop:total_hyp}.

\subsection{Total curvature of surfaces}
For an open complete surface, let us recall the following definition of total curvature. Let $M$ be a $2$-dimensional complete non-compact surface possibly with nonempty boundary, with a Riemannian metric $g$. Denote by $K: M\to \mathbb R$ the Gaussian curvature of $M$, and let $d\mu$ be its associated Riemannian density. Define $K^+=\max\{K, 0\}$ and $K^-=\min\{K, 0\}$. Thus we can define \emph{the total positive curvature} $c_+(M)$ and \emph{total negative curvature} $c_-(M)$ by
\[
c_+(M):=\int_M K^+d\mu,\
\quad
\ c_-(M):=\int_M K^-d\mu.
\]
We will say that the total curvature $c(M)$ of $M$ exists, or that {\em $M$ admits total curvature}, if at least one of $c_+(M)$ or $c_-(M)$ is finite, and we write in that case
$$
c(M):=c_+(M)+c_-(M).
$$
By theorems of \cite{CV1935} and \cite{Hub1966}, the total curvature of $M$ exists if and only if $c_+(M)$ is finite. 

\subsection{The ideal boundary}
Since the general case is rather involved (cf. \cite{Shi1991, Shi1996, SST2003} for general discussion on the ideal boundary of surfaces admitting total curvature), we will focus our study on the special case when $(S, g)$ is homeomorphic to plane $\mathbb R^{2}$, with a complete Riemannian metric $g$ without conjugate points. 

We fix a base point $O\in S$. The unit tangent sphere at $O$ is isometric to a unit circle; we will use the closed interval $[0, 2\pi]$ to parametrize it counterclockwise, with $\theta=0$ and $\theta=2\pi$ corresponding to some given ray $\gamma$ to be fixed later. For any $\theta\in [0, 2\pi]$, there exists a unique ray $\gamma_{\theta}$ emanating from $O$ along the direction $\theta$. We denote the set of rays starting at $O$ by $\Gamma_{O}$. 

For $0\leq\alpha<\beta< 2\pi$, the plane is divided into to parts by the rays $\gamma_\alpha$ and $\gamma_\beta$:
the closed sector $T_{\alpha, \beta}$ containing every point in the rays $\gamma_\theta$ with $\alpha\leq \theta\leq \beta$,  and its closed complement that we denote as $T_{\beta,\alpha}$.

To each one of these sectors, we associate the numbers
\[
\lambda_{\infty}(T_{\alpha, \beta}):=
\beta-\alpha-c(T_{\alpha, \beta}),
\]
and
\[
\lambda_{\infty}(T_{ \beta,\alpha}):=
2\pi-(\beta-\alpha)-c(T_{\beta, \alpha}),
\]
the rule of thumb being that $\lambda_\infty$ of a sector is its interior angle minus its total curvature.

One can define the pseudo-distance 
$$
d_{\infty}(\gamma_{\alpha}, \gamma_{\beta})=\min\{\lambda_{\infty}(T_{\alpha, \beta}), \lambda_{\infty}(T_{\beta, \alpha})\},
$$ 
between rays in $\Gamma_O$;

then the ideal boundary $S(\infty)$ is defined to be the quotient of the set $\Gamma_{O}$ by the equivalent relation $\sim$ given by $\gamma_{\alpha}\sim \gamma_{\beta}$ if and only if $d_{\infty}(\gamma_{\alpha}, \gamma_{\beta})=0$. Then $d_{\infty}$ can be extended to a metric on $S(\infty)$, which is called the Tits metric. The topology induced by $d_{\infty}$ is called the Tits topology.  

Now, we need the following characterization of $S(\infty)$ by Shioya
\begin{prop}[\cite{Shi1991}]\label{prop:shioya1}
Let $S^{2}$ be a simply connected complete Riemannian surface admits total curvature with one end, then
\begin{enumerate}
\item If $\lambda_{\infty}(S)<\infty$, then $(S(\infty), d_{\infty})$ is isometric to the circle of length $\lambda_{\infty}(S)$,
\item If $\lambda_{\infty}(S)=\infty$, then $(S(\infty), d_{\infty})$ is at most continuum union of closed line segments.
\end{enumerate}
\end{prop}

\begin{prop}[\cite{Shi1996}]\label{prop:PropOfLine}
Let $M$ be a Riemannian surface with total curvature. Then for any line $\gamma$, $d_\infty(\gamma(-\infty), \gamma(\infty))\ge \pi.$
Moreover, for any $x, y\in M(\infty)$, with $d_\infty(x, y)> \pi$, there exists a line $\sigma$ with $\sigma(-\infty)=x$ and $\sigma(\infty)=y$. For simplicity, we will say that such line $\sigma$ connects $x$ to $y$.
\end{prop}

\begin{remark}
Shioya's proof of Proposition     \ref{prop:PropOfLine} uses Gauss-Bonnet in the region bounded by two rays asymptotic to $x$ and $y$ and a segment connecting two points on these two rays. Therefore, it will also work in surfaces with boundary. For our application, we will apply \ref{prop:PropOfLine} to a half plane without conjugate points.
\end{remark}

The following result is already known. As mentioned in the introduction, it was proven by Green for compact Riemannian manifolds \cite{Gre1958}; Guimaraes extended it later to  the complete case \cite{Gui1992}.
We include here a new proof for surfaces  that takes advantage of the structure of the ideal boundary. 
\begin{prop}\label{prop:Hopf}
Let $M^{2}$ be an open complete Riemannian surface without conjugate points. Suppose the total curvature of $M$ exists. Then 
$c(M):=\int_{M}Kd\mu\le 0$, and  equality holds if and only if the Riemannian universal cover of $M$ is isometric to the flat $\mathbb R^{2}$.
\end{prop}
\begin{proof}
Let's consider the Riemannian universal cover of $M$, denoted it by $S$. We first note that by Huber's result $c(S)\le 2\pi=2\pi\chi(S)$. Therefore if $c(S)>0$, then by \ref{prop:shioya1} $S(\infty)$ is isometric to a circle of length
$$
\lambda_{\infty}(S)=2\pi\chi(S)-c(S)<2\pi.
$$
It follows that the diameter of $S(\infty)$ is $<\pi$. This contracts to the fact that any line $\gamma$ in $S$ must have $d_{\infty}(\gamma(+\infty), \gamma(-\infty))\ge\pi$. Therefore $c(S)\le 0$. Suppose that if $c(S)=0$, then the ideal boundary of $S$ is isometric to the circle with length $2\pi$. It follows from Theorem 5.2.1 in \cite{SST2003} that 
$$
\lim_{r\to\infty} \frac{2A(r)}{r^{2}}=\lambda_{\infty}(S)=2\pi,
$$
where $A(r)$ is the length of the distance ball $B(O, r)=\{x\in S| |x,O|\le r\}$ for a fixed base point $O\in S$. Applying the rigidity theorem of \cite{BE2013}, we know $S$ is isometric to the flat $\mathbb R^{2}$.
\end{proof}

\subsection{Ray convergence}
The following property is essentially proved in \cite{SST2003}, where all rays are assumed emanating perpendicularly from a core of $M$. For completeness we include a simple proof in our special case. If the surface is a half plane, then there is only one sector between two rays, we called it $T_{\alpha, \beta}$ without referring to any orientation. The distance between $\gamma_{\alpha}$ and $\gamma_{\beta}$ is defined by:
$$
d_{\infty}(\gamma_{\alpha}, \gamma_{\beta})=\lambda_{\infty}(T_{\alpha, \beta}).
$$
Therefore the ideal boundary of the half plane $M$ can be defined in a similar way as the surface $S$. 

\begin{lemma}[Lemma 3.5.1 in \cite{SST2003}]
\label{lem:key}
Let $M$ be a Riemannian surface without conjugate points, homeomorphic to the upper half plane with totally geodesic boundary. Fix a base point $O\in\partial M$. Let $\{\gamma_{\theta_{i}}\}$ be a sequence of rays issued from $O$ which converges to a ray $\gamma_{\theta_{\infty}}$ with $\theta_{i}$ monotonic decreased to $\theta_{\infty}$. Suppose that 
$$
\sup_{i}d_{\infty}(\gamma_{\theta_{1}}, \gamma_{\theta_{i}})<+\infty.
$$ 
Then we have 
\begin{equation*}
d_\infty(\gamma_{\theta_{1}},\gamma_{\theta_{\infty}})<\infty
\end{equation*}
\end{lemma}
\begin{proof}
For simplicity, we write $\gamma_{\theta_{i}}$ as $\gamma_{i}$ and $\gamma_{\theta_{\infty}}$ as $\gamma_{\infty}$. The region bounded by $\gamma_{1}$ and $\gamma_{i}$ in the upper half plane is denoted by $T_{1, i}$ since $\theta_{i}$ is decreasing. Note that in this case, there is only one sector between two rays, we called it $T_{1, i}$ without referring to any orientation. The distance is defined by:
$$
d_{\infty}(\gamma_{1}, \gamma_{i})=\lambda_{\infty}(T_{1, i})
$$
for all $i\in \mathbb{N}$. By the monotonicity of $\{\theta_{i}\}$, we know $T_{1,\infty}$ contains every $\gamma_i$. Denote the open metric ball centered at $O$ with radius $r$ by $B(r)$, and the angle between $\gamma_i$ and $\gamma_{j}$ by $\theta_{i, j}=\theta_{i}-\theta_{j}$ for $i<j$, depicted in Figure \ref{fig:ray}.
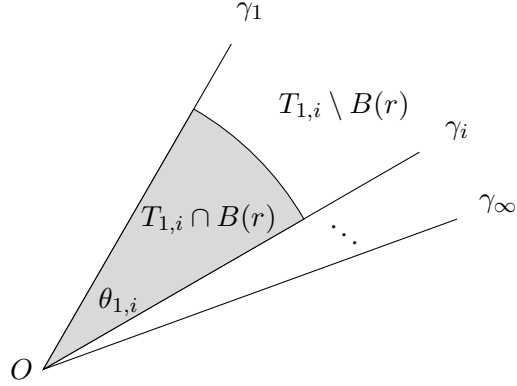
\begin{figure}\label{fig:ray}
\begin{tikzpicture}
   \draw (0,0) node[left]{$O$}--node[pos=1.1]{$\gamma_1$}(2.5,4.33);
   \draw (0,0)--node[pos=1.1]{$\gamma_i$}(5,2.89);
   \draw[fill=gray!30] (0,0)--+(30:4cm) arc (30:60:4cm) -- cycle;
   \draw node at(2.2,2){$T_{1,i}\cap B(r)$};
   \draw node at(4,3.5){$T_{1,i}\setminus B(r)$} ;
   \draw node at(1,0.9){$\theta_{1,i}$};
   \draw node at(4,1.9){$\ddots$};
   \draw (0,0)--node[pos=1.1]{$\gamma_\infty$}(5.5,2); 
\end{tikzpicture}
\caption{Convergent rays in $M(\infty)$}
\end{figure}

By the definition of $d_{\infty}$, for each $i$ and all $r>0$, we have
\begin{equation}\label{eq:key0}
d_\infty(\gamma_1,\gamma_i)=\theta_{1,i}-c(T_{1,i}\cap B(r))-c(T_{1,i}\setminus B(r)).
\end{equation} 
Since
\begin{equation*}
c(T_{1,i}\setminus B(r))\leq \int_{T_{1,i}\setminus B(r)}K^{+}d\mu \leq \int_{T_{1,\infty}\setminus B(r)}K^{+}d\mu.
\end{equation*}
Plugging it into \eqref{eq:key0}, we have
\begin{equation}\label{eq:key1}
d_\infty(\gamma_1,\gamma_i)\geq \theta_{1,i}-c(T_{1,i}\cap B(r))-\int_{T_{1,\infty}\setminus B(r)}K^{+}d\mu.
\end{equation}
Since $B(r)$ is compact, we have
\begin{equation*}
\lim_{i\rightarrow\infty}c(T_{1,i}\cap B(r))=c(T_{1,\infty}\cap B(r)),
\end{equation*}
and clearly
\begin{equation*}
\lim_{i\rightarrow\infty}\theta_{1,i}=\theta_{1,\infty}.
\end{equation*}
Taking lower limit as $i\to \infty$ of the \eqref{eq:key1}, we get
\begin{equation}\label{eq:key2}
\liminf_{i\rightarrow\infty}d_\infty(\gamma_1,\gamma_i)\geq \theta_{1,\infty}-c(T_{1,\infty}\cap B(r)) -\int_{T_{1,\infty}\setminus B(r)}K^{+}d\mu
\end{equation}   
By assumption $S$ admits total curvature. It follows that $\int_{T_{1,\infty}} K^{+} d\mu<\infty$. Therefore
$$
\lim_{r\to \infty}\int_{T_{1,\infty}\setminus B(r)}K^{+}d\mu=0.
$$
Letting $r\to\infty$ on the right hand side of \eqref{eq:key2}, we have
\begin{equation}\label{eq:key3}
\liminf_{i\rightarrow\infty}d_\infty(\gamma_1,\gamma_i)\geq\theta_{1,\infty}-c(T_{1,\infty})=\lambda_{\infty}(T_{1, \infty}).
\end{equation}
The left hand side of \eqref{eq:key3} is finite by the assumption, therefore 
$$
d_\infty(\gamma_1,\gamma_\infty)\le \lambda_{\infty}(T_{1, \infty})<\infty.
$$
This finishes the proof.
\end{proof}

\subsection{Existence of hyperbolic pairs}
\begin{proof}[Proof of Proposition \autoref{prop:total_hyp}]
Now we assume $c(S)<0$, and we will provide a hyperbolic pair for a given $\gamma$.
 
The first case is $c(S)>-\infty$, then $S(\infty)$ is isometric to a circle of length $\ell>2\pi$ by Proposition \autoref{prop:shioya1}. In particular any line $\gamma$ defines two ideal boundary points $\gamma(-\infty)$ and $\gamma(+\infty)$, which are at least $\pi$ apart. Since $\ell >2\pi$, there exists a connected component of $S\setminus \gamma(\mathbb R)$, denoted by $M$ such that $M(\infty)$ is an interval of length larger than $\pi$. Clearly one can choose $x, y\in M(\infty)$ different from $\gamma(-\infty)$ and $\gamma(+\infty)$ and with $d_{\infty}(x, y) >\pi$. Connecting them yields a line $\sigma$. Since our surface $M$ has no conjugate points, $\sigma$ does not intersects $\gamma$. Moreover, 
$$
\infty=\lim_{t\to\infty}d(\gamma_{0}(t), \gamma_{x}(t))\le\lim_{t\to\infty}d(\gamma(t), \sigma(t)),$$ 
and 
$$\infty=\lim_{t\to\infty}d(\gamma_{\pi}(t), \gamma_{y}(t))\le\lim_{t\to-\infty}d(\gamma(t), \sigma(t)),$$
imply that $(\gamma, \sigma)$ is a hyperbolic pair. 

\begin{figure}\label{fig:hyp}
\begin{tikzpicture}
   \draw (0,0)--(8,0)node[right]{$\gamma$};
   \draw (0,2.1)..controls(4, 0.5)..(8, 2.1)node[above]{$\sigma$};
   \draw (4,0)--(0,2)node[below]{$\gamma_x$};
    \draw (4,0)--(8,2)node[below]{$\gamma_y$};
\end{tikzpicture}
\caption{Hyperbolic Pair}\label{HP1}
\end{figure}
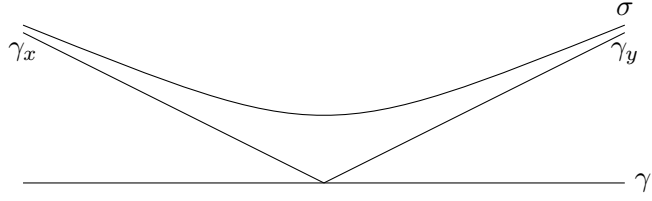

The second case is $c(M)=-\infty$, we have $\lambda_{\infty}(S)=\infty$. Let $\gamma:\mathbb R\to S$ be the given line. We will construct another line $\sigma$ so that the pair $(\gamma, \sigma)$ is a hyperbolic pair. Clearly $\gamma$ divides $S$ into two parts, let consider one of them called $M$ such that $\lambda_{\infty}(M)=\infty$. Now we consider $M$ itself as a Riemannian surface without conjugate points, i.e. we do not view $M$ as a subset of $S$. We can define its ideal boundary $M(\infty)$ as we recalled the definition right before the Lemma  \ref{lem:key}. As before, we fix a base point of  $\gamma(0)\in \partial M=\gamma(\mathbb R)$. Let $\gamma_{0}(t)=\gamma(-t)$ and $\gamma_{\pi}(t)=\gamma(t)$ under the canonical identification of the unit tangent sphere at $O$ and $[0, \pi]$. Consider the set
$$
I_{0}=\{\,\theta\in [0, \pi]\mid d_{\infty}(\gamma_{0}, \gamma_{\theta})<\infty\, \}.
$$
It is easy to see that $I_0$ is an interval: if $\theta\in I_0$, then for any $\theta'$ with $0\leq \theta'\leq \theta$, the sector $T_{0,\theta}$ contains the sector $T_{0,\theta'}$, and thus from $c(T_{0,\theta})>-\infty$, we get $c(T_{0,\theta'})>-\infty$.
 
Clearly $0\in I_{0}$, and $\pi\not\in I_{0}$. There are two sub-cases; the first is when $I_{0}$ is a relatively open interval in $[0, \pi]$. i.e. 
$$
I_{0}=[0, a), \quad a<\pi.
$$
This means 
$$\lim_{\theta\to a^{-}}d_{\infty}(\gamma_{0}, \gamma_{\theta})= \infty$$
That is there exists $[x, y]\subset (0, a)$ such that 
$$d_{\infty}(\gamma_{0}, \gamma_{x})>0, \quad {\rm and}\quad d_{\infty}(\gamma_{x}, \gamma_{y})>\pi$$ 
Then by Proposition \autoref{prop:PropOfLine}, we can connect $\gamma_{x}$ to $\gamma_{y}$ to get a line $\sigma$. Same argument as above shows $(\gamma, \sigma)$ is a hyperbolic pair.
Similarly, we can define 
$$
I_{\pi}=\{\,\theta\in [0, \pi] \mid d_{\infty}(\gamma_{\theta}, \gamma_{\pi})<\infty \,\}.
$$
The very same argument shows that if $I_{\pi}=(b, \pi]$ then we can find a hyperbolic pair.  

So we only need to consider the second sub-case, i.e.  $I_{0}=[0, a]$ and $I_{\pi}=[b, \pi]$ with $a<b$. By the construction of $I_{0}$ and $I_{\pi}$ we know 
$$
d_{\infty}(\gamma_{0}, \gamma_{c})=\infty=d_{\infty}(\gamma_{\pi}, \gamma_{c}),
$$
where $c=\frac{a+b}{2}$. We consider
$$
I_{c}=\{\theta\in [0, \pi]\mid d_{\infty}(\gamma_{\theta}, \gamma_{c})<\infty\}.
$$
Clearly by Lemma \ref{lem:key}, $I_{c}$ is a sub interval of $(a, b)$, with end points $a'\le c\le b'$. If one of the end point of $I_{c}$ is open, then the same construction as in the sub-case $I_{0}=[0, a)$ will give two points in the ideal boundary of distance at least $1.5\pi$ apart in $M(\infty)$, therefore connecting them will give us a desired $\sigma$. If both end points are closed. i.e.
$$
I_{c}=[a', b'],
$$
with $a<a'$ and $b'<b$. Then we let $x=\frac{a+a'}{2}$ and $y=\frac{b'+b}{2}$. Then
$$
d_{\infty}(\gamma_{0}, \gamma_{x})=d_{\infty}(\gamma_{x}, \gamma_{y})=d_{\infty}(\gamma_{y}, \gamma_{\pi})=\infty.
$$
Connecting $x$ to $y$ using \autoref{prop:PropOfLine} yields the desired $\sigma$.
\end{proof}

\section{Visibility planes and weak hyperbolic pairs}
\begin{definition}
A Riemannian manifold $M$ without conjugate points satisfies the \emph{visibility axiom at $p$} if for each $\epsilon>0$ there exists a constant $R=R(\varepsilon,p)>0$ such that for any geodesic segment $\sigma:[a,b]\to M$ satisfying $d(p,\sigma)>R$, the angle $\measuredangle_p(\sigma(a),\sigma(b))\leq \varepsilon$. 
\end{definition}

If $p$ is arbitrary, such a manifold is called a \emph{visibility manifold}; the reader can consult \cite{Eber1972} or \cite{EberOnei1973} for several of their properties.

\begin{defn}
Two geodesic rays $\gamma$, $\sigma:[0,\infty)\to M$ are \emph{asymptotes} if there exists a constant $C>0$ such that $d(\gamma(t),\sigma(t))\leq C$ for every $t>0$.
\end{defn}

We will write 
$\gamma\sim \sigma$ to indicate that two rays are asymptotes.
Due to the triangle inequality, $\sim$ is an equivalence relation on the set of rays in $M$.

\begin{defn}
Let $\Gamma_{p}$ be the set of all geodesic rays in $M$ emanating from $p$. We will use  $M_a(\infty)$ to denote the set of equivalence classes $\Gamma_{p}/\sim$. If $\gamma\in \Gamma_{p}$, we will denote by $\gamma(\infty)$ its equivalence class in $M_a(\infty)$. 
\end{defn}

\begin{prop}\label{prop:boundary}
Let $M$ be a open manifold satisfying the visibility axiom at some point $p$. Then for any two different rays $\gamma$ and $\sigma$ emanating from $p$, we have that $\gamma(\infty)\ne \sigma(\infty)$.
\end{prop}
\begin{proof}
Suppose the contrary that there exists some $c>0$ such that $d(\gamma, \sigma)<c$. Then $\measuredangle_{p}(\gamma(t), \sigma(t))\to 0$ since the segments connecting $\gamma(t)$ and $\sigma(t)$ clearly drifts to infinity as $t\to\infty$. On the other hand $\measuredangle_{p}(\gamma(t), \sigma(t))=\measuredangle(\gamma'(0), \sigma'(0))$, a contradiction.
\end{proof}

\begin{prop}\label{prop:line}
Let $x, y\in M_a(\infty)$ be two different points, and denote by $\gamma_x$ and $\gamma_y$ respectively the rays emanating from $p$ with $\gamma_x(\infty)=x$ and $\gamma_y(\infty)=y$. Then for any sequence $t_k\to\infty$, the sequence of the segments $\alpha_{k}$ connecting $\gamma_{x}(t_k)$ and $\gamma_{y}(t_k)$  has a subsequence that converges to a line in $M$.
\end{prop}
\begin{proof}
Since $\measuredangle_{p}(\gamma_{x}(t_k), \gamma_{y}(t_k))$ equals to $\measuredangle(\gamma_{x}'(0), \gamma_{y}'(0))=\varepsilon_0>0$, it follows from the definition of visibility at $p$ that there exists an $R=R(p, \varepsilon_0)$ such that
$$
d(p, \alpha_{k})\le R,
$$
for $k\in \mathbb N$. Since closed balls of bounded radius are compact, the sequence of segments  $\{\alpha_{k}\}_{k=1}^{n}$ subconverges to a line.
\end{proof}
  
If $M$ is a plane, then the limiting line $\alpha$ stays inside the open convex region bounded by the rays $\gamma_{x}$ and $\gamma_{y}$. Since this is the 2-dimensional case, we will switch from $M$ to $S$ to denote it in what follows.

\begin{proof}[Proof of Proposition \ref{prop:visi_hyp}]
We will see that, given $\gamma\in\mathcal{F}$, there is some line $\sigma$ such that $(\gamma,\sigma)$ is a weak hyperbolic pair.
$\gamma$ separates $S$ into two open half planes. Let $\alpha_{+}$ and $\alpha_{-}$ be two different rays emanating from $p=\gamma(0)$ and lying within one component. By Proposition \ref{prop:line}, we can construct a line $\sigma$ which lies within the sector bounded by $\alpha_{\pm}$. On the other hand, from Proposition \ref{prop:boundary}, $\gamma:[0,\infty)\to S$ and $\alpha_+$ are not asymptotes. Thus, there is a sequence $t_k\to\infty$ such that
\[
d(\gamma(t_k),\alpha_+(t_k))\to\infty
\] 
  and therefore
$$d(\sigma(t_k), \gamma(t_k))>d(\alpha_{+}(t_k), \gamma(t_k))\to\infty.$$ 
A similar argument for $\sigma$ and $\gamma$ at $-\infty$ proves that $(\gamma, \sigma)$ is a weak hyperbolic pair. 
\end{proof}

\section{Proofs of Theorems
\ref{thm:flat} and \ref{thm:visib_foliation}}
\begin{proof}[Proof of Theorem \ref{thm:flat}]
Let $\gamma \in \mathcal{F}$ be a line. Suppose $c(S)<0$; then by Proposition \ref{prop:total_hyp}, there exists a line $\sigma$ such that $(\gamma, \sigma)$ is a hyperbolic pair. However this is impossible. In fact for any point $x_{0} \not\in S_{\gamma}^{\sigma}$ and $x\in H_{\sigma}^{+}$, there must be a leave $\bar\gamma$ passing through it and with bounded distance to $\gamma$, see Figure \ref{HP}. Let
$$
C=\dist_{H}(\bar\gamma, \gamma),
$$
be the Hausdorff distance between the two leaves, and parametrize $\gamma$ and $\bar\gamma$ such that $\dist(\gamma(0), \bar\gamma(0))\leq C$. Then after possible reversing the parameter of one of the lines, we have
$$
\lim_{t\to\infty}\dist(\gamma(t), \bar\gamma(t))\le 3C.
$$
In fact, for any $t$, there exists $s$ such that $\dist(\gamma(t), \bar\gamma(s))\le C$ by the definition of Hausdorff distance. Note that the $\gamma$ and $\bar\gamma$ are lines, therefore any two points on one line must realize the distance between them, it follows from the triangle inequality that 
$t\leq s+2C$, and $s\leq t+2C$, thus $|s-t|\leq 2C$.
 Therefore 
$$\dist(\gamma(t),\bar\gamma(t))\le \dist(\gamma(t),\bar\gamma(s))+|t-s|\le 3C.$$
By the definition of hyperbolic pair, $\bar\gamma$ must exit $ H_{\sigma}^{+}$, so $\bar\gamma$ must intersect $\sigma$ twice, this contradicts to the fact that geodesics do not bifurcate in Riemannian manifold. Therefore it follows that $c(S)=0$ and by Proposition \autoref{prop:Hopf}, $S$ is isometric to the flat $\mathbb R^{2}$.
\end{proof}

\begin{figure}\label{fig:hyp_proof}
\begin{tikzpicture}
\shade [top color=gray!10, bottom color=gray!20](0,0)[dashed]--(0,2)--(0,2)arc(-131.4:-48.6:6cm)--(8,0)--cycle;
   \draw (0,0)--(8,0)node[right]{$\gamma$};
   \draw (0,1)--(8,1)node[right]{$\bar\gamma$};
   \draw (4,0.5) arc (-90:-48.6:6cm)node[right]{$\sigma$};
   \draw (4,0.5) arc (-90:-131.4:6cm);
   \draw node at (4,1)[circle,fill,inner sep=0.1pt]{};
   \draw node at (4,1)[above]{$x_{0}$};
   \draw node at(1.5,0.5){$H_{\gamma}^{\sigma}$};
   \draw node at(4.5,2){$H_{\sigma}^{+}$};
\end{tikzpicture}
\caption{Proof of Theorems \ref{thm:flat} and \ref{thm:visib_foliation}}
\label{HP}
\end{figure}
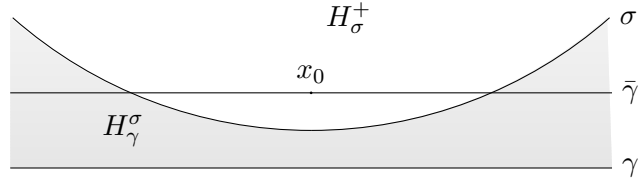

\begin{proof}[Proof of Theorem  \ref{thm:visib_foliation}]
The argument for the first part of the Theorem mimics the previous proof, but using Proposition \ref{prop:visi_hyp} to replace the hyperbolic pair by a weak hyperbolic pair, and working with a sequence $t_k\to\infty$ instead of arbitrary $t\in\R$.  

For the second part of the Theorem, suppose $M$ is the universal cover of a closed surface $\Sigma_{g}$ with genus $g\ge 1$ and without conjugate points.
If $g=1$, then by E. Hopf's theorem for torus without conjugate points, $\Sigma_g$ is a flat torus, and $M$ is the flat plane.  
If $g>1$, then $M$ satisfies the visibility axiom, cf. \cite{Eber1972}. Therefore from the first part of the Theorem, it cannot admit a line foliation with bounded Hausdorff distance.  
\end{proof}

\bibliographystyle{alpha}
\bibliography{ref}
\end{document}